\documentclass[12pt, reqno]{amsart}

\usepackage[utf8]{inputenc}
\usepackage[T1]{fontenc}
\usepackage{tikz,tikz-3dplot}
\usepackage{amsmath,amsfonts,amssymb,amsmath,latexsym,mathrsfs}
\usepackage{shuffle}
\usepackage[all]{xy}
\usepackage{stackengine}
 
\usepackage{enumerate}

\usepackage{hyperref}
 
\usepackage{tabu}
\usepackage{tikz-cd} 

\usepackage{comment}
\usepackage{url}
\usepackage{tikz}

\usepackage{xfrac}

\newtheorem{theorem}{Theorem}[section]

\newtheorem{proposition}[theorem]{Proposition}
\newtheorem{definition}[theorem]{Definition}

\newtheorem{corollary}[theorem]{Corollary}

\newtheorem{remark}[theorem]{Remark}

\newtheorem{lemma}[theorem]{Lemma} 

\newtheorem{example}[theorem]{Example} 

\newcommand{\Z}{ {\mathbb Z} }
\newcommand{\C}{ {\mathbb C} }

\begin{document}

\title{On irreducible representations of conjugacy quandles}

\author{Mohamad \textsc{Maassarani} }
\maketitle
\begin{abstract}For $G$ a finite group, one way to construct irreducible quandle representations over $\mathbb{C}$ of the conjugacy quandle $Conj(G)$ is by taking the product of an irreducible linear group representation of $G$ by what we call a quandle character of $Conj(G)$ (a quandle morphism into $\mathbb{C}^\times$ ). We show that these are all the irreducible quandle representations of $Conj(G)$ over $\mathbb{C}$ if and only if the Bogomolov multiplier $B_0(G)$ of $G$ is trivial. The Bogomolov multiplier is trivial for many classical groups, simple groups and with exceptions quasi-simple groups. We apply this to study the enveloping group of $Conj(G)$. If $G$ finite with trivial $B_0(G)$, we obtain that the enveloping group of $Conj(G)$ injects into $G\times \mathbb{Z}^{c_G}$ where $c_G$ is the number of the conjugacy classes of $G$. If moreover $G$ is perfect the injection is an isomorphism.
\end{abstract}
\section*{Introduction and main results}
A quandle is set equiped with a binary operation satisfying some axioms. To a group $G$ one can associate a quandle $Conj(G)$ called the conjugacy quandle of $G$. The underlining set of $Conj(G)$ is $G$. A quandle representation of $Conj(G)$ is a quandle morphism $\rho : Conj(G)\to Conj(GL(V))$ ($\rho(hgh^{-1})=\rho(h)\rho(g)\rho(h)^{-1}$) for some vector space $V$. A quandle representation $Conj(G)\to Conj(GL(V))$ is irreducible if the image of the morphism stabilizes non subspace of $V$ other than $V$ and $0$. In \cite{GQ}, we proved than an irreducible quandle representation over $\mathbb{C} $ ($V$ is a complex vector space) of a finite quandle is finite dimensional. In \cite{Qrep}, we prove for some finite groups, the irreducible quandle representation of $Conj(G)$ over $\mathbb{C}$ are the products of what we call quandle characters of $Conj(G)$ and irreducible linear representations of the group $G$. Here we prove that all irreducible quandle representations of $Conj(G)$ over $\mathbb{C}$ ($G$ finite) are of this form if and only if any symmetric $2$-cocycle of $G$ with values in $\mathbb{C}^\times$ is a coboundary. This condition turns out to be equivalent to the triviality of the Bogomolov multiplier.  We apply these results to study the envelopping groups of conjugacy quandles. We prove for instance, that if $G$ finite is perfect and has trivial Bogomolov multiplier then the enveloping group of $G$ is isomorphic to $G\times \Z^{c_G}$, where $c_G$ is the number of conjugacy classes of $G$. The notes are divided into $3$ sections. In the first section, we recall the basic definitions. In the second section, we provide proofs of the results concerning irreducible quandle representations. The third section is devoted to the study of envelopping groups of conjugacy quandles.

\section{Reminders on quandles and projective representations of groups}
\subsection{Quandles}
\begin{definition}
 A quandle is a set $Q$ equipped with a binary operation $\triangleright$  such that :
\begin{itemize}
\item[-] $x\triangleright x=x$ for $x\in Q$.
\item[-] For all $x,y \in Q$ there exist a unique $z\in Q$, such that $x \triangleright z=y$.
\item[-] $x \triangleright( y\triangleright z)= (x \triangleright y) \triangleright (x \triangleright z)$, for $x,y,z \in Q$. 
\end{itemize}
\end{definition}
One can constructe a quandle from a group $G$ using the binary operation defined by $x\triangleright y= xyx^{-1}$  for $x,y\in G$. The defined quandle is called the conjugacy quandle and is usually denoted by $Conj(G)$.

\begin{definition}
 A quandle morphism is a map between two quandles $(Q_1,\triangleright_1)$ and $(Q_2 ,\triangleright_2)$ such that $f(z \triangleright_1 w)=f(z) \triangleright_2 f(w)$, for $z,w \in Q$.
\end{definition}
The composition of two quandle morphism is a quandle morphism. A group morhism $f:G\to H$ is quandle morphism if we regard both sets as the conjugacy quandles. For $Q$ a quandle the left translation $L_x$ by $x\in Q$ given by $L_x(y)=x\triangleright y$ for $y\in Q$ is a bijective quandle morphism  
\begin{definition}
The inner automorphism group $Inn(Q)$ of a quandle $Q$ is the subgroup of bijections of $Q$ generated by the left translations.
\end{definition}
\begin{proposition}
The map $Q\to Inn(Q), x\mapsto L_x$ is a quandle morphism if $Inn(Q)$ is regarded as the conjugacy quandle of the group $Inn(Q)$.
\end{proposition}  
For $G$ a group, the group $Inn(Conj(G))$ is equal to $Inn(G)$.

\begin{definition}{\cite{rep}}Let $Q$ be a quandle
\begin{itemize}
\item[1)] A linear representation of $Q$ is the data of a vector space $V$ and a quandle morphism $\rho: Q \to Conj(\mathrm{GL}(V))$, i.e. : 
$$\rho(x\triangleright y)=\rho(x) \rho(y) \rho(x)^{-1},$$
for all $x,y \in Q$.
\item[2)] A subrepresentation of $\rho$ is a vector subspace $W\subset V$, such that $\rho(x)(W) \subset W$ for all $x\in Q$.
\item[3)] An irreducible representation is a representation $V$ that has no subrepresentations other than $0$ and $V$.
\end{itemize}
\end{definition} 
A linear group representation $\rho : G \to GL(V)$ is a quandle representation of $Conj(G)$.
\begin{definition}{\cite{Qrep}}
A character of a quandle $Q$ is a quandle morphism $\chi : Q \to \mathbb{C}^\times$, i.e. :$$\chi(x\triangleright y) =\chi(y),$$ for $x,y\in Q$.
\end{definition}
\begin{definition}
The enveloping group $G(Q)$ of a quandle $Q$ is the group :
$$G(Q)=\langle x \in Q \vert xyx^{-1}=x\triangleright y , \text{ for } x,y\in Q\rangle.$$
\end{definition}
The enveloping group is also known as : the associated group, adjoint group or structure group. The group $G(Q)$ is infinite.
\begin{proposition}
 
\item[1)] The map $\varphi_Q : Q \to G(Q)$ mapping $x\in Q$ to the corresponding generator is a quandle morphism for the conjugcay quandle structure on $G(Q)$.
\item[2)]For $f:Q\to Conj(G)$ a quandle morphism there is a unique group morphism $\tilde{f}:G(Q)\to G$ such $f=\tilde{f}\circ\varphi_Q$. The map $\varphi_Q: Q \to G(Q)$ is universal with respect to morphism from $Q$ to conjugacy quandles of groups.
 
\end{proposition}
\subsection{Projective representations}
 In this subsection, the groups are \textbf{finite} and representations are \text{over} $\mathbb{C}$.
\begin{definition}
\item[1)]  A projective representation of a group $G$ is a morphism $G\to PGL(V)$.
\item[2)]  A projective representation $G \to PGL(V)$ is irreducible if $P(V)$ admit no proper projective subspace stable by $G$.
\end{definition} 
Let $\rho : G \to PGL(V)$ be a projective representation and $\tilde{\rho}$ be a lift of $\rho$ to $GL(V)$ such that $\tilde{\rho}(1)=Id$. For $g,h\in G$, we have:
$$\tilde{\rho}(gh)=\alpha(g,h)\tilde{\rho}(g)\tilde{\rho}(h),$$
for some $\alpha(g,h)\in \mathbb{C}^\times$. The mapping $\alpha :G\times G\to \mathbb{C}^\times,(g,h)\mapsto \alpha(g,h)$ is a $2$-cocycle with values in $\mathbb{C^\times}$. If $\tilde{\rho}'$ is another lift of $\rho$ to $GL(V)$ such that $\tilde{\rho}'(1)=Id$, then the cocycles $\alpha$ and the one associated to $\tilde{\rho}'$ will differ by a boundry. This shows that $\rho$ definies a cohomology class in $H^2(G,\mathbb{C}^\times)$ . The group $H^2(G,\mathbb{C}^\times)$ is called the \textbf{Schur multiplier} of $G$. The Schur multiplier of a finite group is finite.
\begin{proposition}
 A projcetive representation $\rho :G \to PGL(V)$ lifts to a linear representation $G\to GL(V)$ if and only if its class in $H^2(G,\mathbb{C}^\times)$ is trivial.
\end{proposition}
A proof of the following can be found in \cite{Sc} (p. 40-42) :
\begin{proposition}
To each class $\alpha$ of the Schur multiplier $H^2(G,\mathbb{C}^\times)$ there is at least one irreducible projective representation of $G$ associated to $\alpha$.
\end{proposition}

\section{irreducible representations of quandles}
Let $Q$ be a quandle.  If $\rho : Q \to GL(V)$ be a quandle representation then by the universal property of $ G(Q)$ there is a unique group representation $\rho_{G(Q)}:G(Q)\to GL(V)$ such that $\rho=\rho_{G(Q)}\circ \varphi_Q$. 
.
\begin{proposition} 
\item[1)] The assignement $\rho \to \rho_{G(Q)}$ is a one to one correspondance between quandle representations of $Q$ and group representations of $G(Q)$.
\item[2)] A quandle representation $\rho$ is irreducible if and only if the group representation $\rho_{G(Q)}$ is irreducible.
 
\end{proposition}
 
For a group $G$, we will use the following \textbf{notations} :
$$A(G):=G(Conj(G)),$$
$$\varphi_{G}:=\varphi_{Conj(G)}.$$
With this notations $\varphi_G$ is the universal map $G\to A(G)$. We recall that it is a quandle morphism. The identity map  $G\to G$ is a quandle morphism with respect to the conjugacy quandle structure of $G$. Hence we have a unique \textbf{group morphism} $\pi :A(G) \to G$ such that the following diagram commutes :

$$\begin{tikzcd}
A(G) \arrow{r}{\pi} & G  \\ 
G\arrow{u}{\varphi_G} \arrow{ru}{\mathrm{id}} & 
\end{tikzcd} .$$ 

 \begin{proposition}
The kernel of $\pi : A(G)\to G$ is a subgroup of the center of $A(G)$.
\end{proposition}
\begin{proof}
Let $\alpha :G \to Inn(G)$ be the map assigning to $g$ the conjugacy by $g$. We have that $Inn(G)=Inn(Conj(G))$ and $\alpha$ is exactly the quandle morphism $\theta :Conj(G)\to Inn(Conj(G))$ mapping  $g$ to the left translation (in $Conj(G)$) by $g$. The map $\theta$ extends to a unique group morphism $\tilde{\theta} :A(G) \to Inn(Conj(G))$. By uniqueness $\tilde{\theta}=\alpha\circ \pi $. Hence, the kerenel of $\pi$ is a subgroup of the kernel of $\tilde{\theta}$. It is know that the kernel of $\tilde{\theta}$ lies in the center of $A(G)$ (\cite{EM}).
\end{proof}
\begin{proposition}{\cite{GQ}}
For $Q$ a finite quandle, an irreducible representation of $Q$ $($respectively of $G(Q)$$)$ over $\mathbb{C}$ is finite dimensional.
\end{proposition}
From now on $G$ is a \textbf{finite} group and $\rho : G \to Conj(GL(V))$ is an \textbf{irreducible} quandle repesentation \textbf{over} $\C$. In particular $V$ is finite dimensional. All the represenatations we will consider are \textbf{over} $\mathbb{C}$. Let $s$ be a section of $\pi:A(G)\to G$. Define $\bar{\rho}: G \to PGL(V)$ by :  
$$\bar{\rho}(g)=\overline{\rho_{A(G)}(s(g))},$$
where $\rho_{A(G)}$ is the group representation of $A(G)$ induced by $\rho$ and $\overline{\rho_{A(G)}(s(g))}$ is the class of $\rho_{A(G)}(s(g))$ in $PGL(V)$ for $g\in G$. 
\begin{proposition}\label{pro}
\item[1)] $\bar{\rho}$ is an irreducible projective representation of $G$ and does not depend on the choice of the section $s$.
\item[2)] For $g\in G$, $\bar{\rho}(g)=\overline{\rho(g)}$, where $\overline{\rho(g)}$ is the class of $\rho(g)$ in $PGL(V)$.
 \end{proposition}
\begin{proof}
$1)$ follows from the fact that the kerenel of $\pi$ lies in the center of $A(G)$ and hence acts by scalars in the irreducible representation $V$. To obtain $2)$ for a given $g$, we can take $s$ such that $s(g)=\varphi_G(g)$.
\end{proof}
Let $\chi$ be a character of $Conj(G)$. We define the product $\chi \cdot \rho$ of $\rho$  by $\chi$, by the following : 
$$ \chi \cdot \rho (g)=\chi(g) \rho(g),$$
for $g \in G$. The quandle representation $\chi \cdot \rho$ is irreducible. Moreover it follows from $2)$ of the previous proposition that $\overline{\chi \cdot \rho}=\bar{\rho}$.
\begin{proposition}\label{pro1}
\item[1)] If $\rho':G\to Conj(GL(V))$ is an irreducible quandle representation such that $\bar{\rho}'=\bar{\rho}$ then $\rho'=\chi \cdot \rho$ for some character $\chi$ of $Conj(G)$.
\item[2)] Assume that $\bar{\rho}$ lifts to a linear group represenation $\tilde{\rho}:G\to GL(V)$, then $\tilde{\rho}$ is irreducible and $$\rho=\chi \cdot \tilde{\rho},$$ for some character $\chi$ of $Conj(G)$.
\end{proposition} 
\begin{proof}
We prove $1)$. If $\bar{\rho}'=\bar{\rho}$ then for $g,h \in G$, we have :
$$ \rho'(g)=\chi(g)\rho(g), \quad \rho'(h)=\chi(h) \rho(h) \quad\text{ and } \rho'(g\triangleright h)=\chi(g\triangleright h) \rho(g\triangleright h),$$
for some $\chi(g),\chi(h),\chi(g\triangleright h)\in \mathbb{C}^\times$. Now :
$$\rho'(g\triangleright h)=\rho'(g)\rho'(h)\rho'(g)^{-1}=\chi(h)\rho(g)\rho(h)\rho(g)^{-1}=\chi(h) \rho(g\triangleright h).$$
This proves that $\chi(x\triangleright y)\rho(x\triangleright y)=\chi(y)\rho(x\triangleright y)$ and $1)$ follows.\\We prove $2)$. The representation $\tilde{\rho}$ is irreducible since $\bar{\rho}$ is irreducible. It is clear that $\bar{\tilde{\rho}}=\bar{\rho}$. Hence by $1)$ of this proposition $\rho=\chi \cdot \tilde{\rho}$ for some character of $Conj(G)$.
\end{proof}
\begin{definition}
A $2$-cocycle $\alpha :G\times G\to \mathbb{C}^\times$ is symmetric if :$$\alpha(g,h)=\alpha(h,g),$$ for all $g,h\in G$.
\end{definition}
\begin{lemma}
A $2$-cocycle $\alpha: G\times G\to \mathbb{C}^\times$ satisfies the eqaution :
$$\alpha(g,hg^{-1})\alpha(h,g^{-1})\alpha(1,1)^{-1}\alpha(g,g^{-1})^{-1}=1.$$
if and only if it is symmetric.
\end{lemma}
\begin{proof}
By replacing $h$ by $hg$ in the equation, we get the equivalent equation
$$\alpha(g,h)\alpha(hg,g^{-1})\alpha(1,1)^{-1}\alpha(g,g^{-1})^{-1}=1.$$
By the cocycle condition : 
$$\alpha(hg,g^{-1}) \alpha(h,g)\alpha(h,1)^{-1}\alpha(g,g^{-1})^{-1}=1.$$
Combining both equation we get the equivalent equation to the one in the lemma:
$$ \alpha(g,h)\alpha(1,1)^{-1}\alpha(h,g)^{-1}\alpha(h,1)=1.$$
But by the cocycle condition, $\alpha(h,1)\alpha(1,1)\alpha(h,1)^{-1}\alpha(h,1)^{-1}=1$ and hence $\alpha(h,1)=\alpha(1,1)$. This simplifies the last equation to :
$$ \alpha(g,h)=\alpha(h,g). $$
We have proved the lemma.
\end{proof}
\begin{lemma}\label{lem}
Let $\theta : G \to PGL(V)$ be an irreducible projective representation of $G$. There is an irreducible quandle representation $\rho:G\to Conj(GL(V)$ inducing $\theta$ (i.e. $\bar{\rho}=\theta$) if and only if the class of $\theta$ in $H^2(G,\mathbb{C}^\times)$ admits a symmetric representative. 
\end{lemma}
\begin{proof}
Let $\theta$ be a an irreducible projective represenation with class $[\theta]\in H^2(G,\mathbb{C}^\times$. Now any lift (as a map) $\gamma : G \to GL(V)$ of $\theta$ satisfies :
$$\gamma(gh)=\alpha(g,h)\gamma(g)\gamma(h),$$
for $\alpha : G\times G\to \mathbb{C}^\times$ a $2$-cocycle representing $[\theta]$ and any $2$-cocycle representing $[\theta]$ is associated to some lift. A lift $\gamma$ is a quandle representation if and only if $\gamma(ghg^{-1})=\gamma(g)\gamma(h)\gamma(g)^{-1}$. Now in term of the cocycle : 
$$\gamma(ghg^{-1})=\alpha(g,hg^{-1})\alpha(h,g^{-1})\gamma(g)\gamma(h)\gamma(g^{-1}),$$
$$\gamma(1)=\alpha(g,g^{-1})\gamma(g)\gamma(g^{-1}),$$
$$\gamma(1)=\alpha(1,1)^{-1}\mathrm{id}.$$
Combining these equations we get :
$$\gamma(ghg^{-1})=\alpha(g,hg^{-1})\alpha(h,g^{-1})\alpha(g,g^{-1})^{-1}\alpha(1,1)^{-1}\gamma(g)\gamma(h)\gamma(g)^{-1}.$$
This proves that a lift $\gamma$ is a quandle representation if and only if :
$$\alpha(g,hg^{-1})\alpha(h,g^{-1})\alpha(g,g^{-1})^{-1}\alpha(1,1)^{-1}=1,$$
wich is equivalent to $\alpha$ symmetric by the previous lemma. What we have seen proves that $\theta$ lifts to a quandle representation if and only if the class $[\theta]$ admits a symmetric representative. It is not hard to check that the lift is irreducible.
\end{proof}
\begin{remark}
If $\rho :G \to Conj(GL(V)$ is an irreducible quandle representation then for any character $\chi$ of $Conj(G)$ the quandle representation $\chi \cdot \rho$ is irreducible. 
\end{remark}
Denote by $H_S^2(G,\mathbb{C}^\times)$ the subgroup of $H^2(G,\mathbb{C}^\times)$ consisting of symmetric $2$-cocycles.
\begin{theorem}\label{thsym}
The irreducible quandle representations of a finite group $G$ over $\mathbb{C}$ are the representations of the form $\chi \cdot \rho$ for $\chi$ a character of $Conj(G)$ and $\rho$ an irreducible linear group representation of $G$, if and only if $H_S^2(G,\mathbb{C}^\times)=0$.
\end{theorem}
\begin{proof}
For $\chi \cdot \rho$ as in the theorem $\overline{\chi\cdot \rho}=\bar{\rho}$. Hence, if the irreducible representations of $G$ are as in the theoreim, the class of the projective representation induced by an irreducible quandle representation is zero. It follows from the previous lemma, that in this case, every symetric $2$-cocycle has a trivial class in $H^2(G,\mathbb{C}^\times)$ and hence is a coboundary. Now assume that every symmetric $2$-cocyle is a coboundary. It follows from the previous lemma that the only irreducible projective representation that lift to an irreducible quandle representation in fact lift to an irreducible linear group representation of $G$. It follows from $2)$ of \ref{pro1} that any irreducible quandle representation of $G$ is of the form $\chi \cdot \rho$ as in the theorem.
\end{proof}
\begin{remark}
Considering the determinant one shows that the representations $\chi \cdot \rho$ of the proposition take values in the unitary group if and only if $\rho$ is unitary and $\chi$ takes values in complex numbers with modulus $1$.
\end{remark}

\begin{definition}
For $G$ a finite group we define $B_\mathbb{C}(G)$ as the subgroup of $H^2(G,\mathbb{C}^\times)$ consisting of $2$-cocycle classes that are trivial after restriction to any abelian subgroup of $G$. 
\end{definition}
The following result is known. We provide a proof for completness.
\begin{proposition}
If $H$ is abelian and $\alpha :H\times H \to \mathbb{C}^\times$ is a symmetric $2$-cocycle, then $\alpha$ is a coboundary.
\end{proposition}
\begin{proof}
For $H$ abelian the envelopping group $A(H)$ is abelian. Hence, the irreducible representations of $Conj(H)$ are $1$-dimensional. Therefore, by the last lemma any symmetric $2$-cocycle is associated to a projective representation $\rho :H \to PGL(V)$ with $V$ one dimensional. Taking any lift $\tilde{\rho}$ (as a map) of $\rho$ to $GL(V)$, we get :
$$ \tilde{\rho}(gh)=\alpha(g,h) \tilde{\rho}(g)\tilde{\rho(h)},$$
were $\alpha$ is a representative of the class of $\rho$ in $H^2(H,\mathbb{C}^\times)$. Applying an isomorphism $GL(V)\to \mathbb{C}^\times$ to the equation, we get that $\alpha$ is a coboundary.  
\end{proof}
\begin{corollary}
For $G$ a finite group, the class of a symmetric $2$-cocyle of $G$ with values in $\mathbb{C}^\times$ in $H^2(G,\mathbb{C}^\times)$ lies in $B_\mathbb{C}(G)$.
\end{corollary}
 
\begin{definition}
For $G$ a finite group, the Bogomolov multiplier $B_0(G)$ is the subgroup of $H^2(G,\mathbb{Q}/\mathbb{Z})$ consisting of $2$-cocycle classes that are trivial after restriction to any abelian subgroup of $G$. 
\end{definition}
Let $\phi : \mathbb{Q}/\mathbb{Z}\to \mathbb{C}^*$ be the injective group morphism induced by $x\mapsto e^{2i\pi x}$ for $x\in \mathbb{Q}$. To a $2$-cocycle $\alpha : G \times G \to \mathbb{Q}/\mathbb{Z}$ one can assign the $2$-cocycle with values in $\mathbb{C}^\times$ defined by $\phi \circ \alpha$. This induces a morphism $\phi_*^G : H^2(G, \mathbb{Q}/\mathbb{Z})\to H^2(G,\mathbb{C}^\times)$. 
\begin{proposition}
The morphism $\phi_*^G : H^2(G, \mathbb{Q}/\mathbb{Z})\to H^2(G,\mathbb{C}^\times)$ is an isomorphism.
\end{proposition}
\begin{proof}
We first prove that the morphism is surjective. The $2$-cocycles with values in $\phi(\mathbb{Q}/\mathbb{Z})$ are in the image of $\phi_*^G$. We will prove that any class in $H^2(G,\mathbb{C}^\times)$ in represented by a $2$-cocycle with values in $\phi(\mathbb{Q}/\mathbb{Z})$. To a class $x$ in $H^2(G,\mathbb{C}^\times)$ correspond an irreducible projective representation $\rho : G \to PGL_n(\mathbb{C})$. Let $\tilde{\rho}:G\to SL_n(\mathbb{C})$ be a map lifting $\rho$. We have :
$$\tilde{\rho}(gh)=\alpha(g,h)\tilde{\rho}(g)\tilde{\rho}(h)$$
with $\alpha$ a $2$-cocycle representing the class $x$. Since the lifts are taken in $SL_n(\mathbb{C})$, we get by taking the determinant that $\alpha(g,h)^n=1$. Hence, the $2$-cocyle $\alpha $ is with values in $\phi(\mathbb{Q}/\mathbb{Z})$ and represent $x$. This proves that $\phi_*^G$ is surjective. By the universal coefficient theorem, we get : 
$$ H^2(G,\mathbb{C}^\times)\simeq Hom(H_2(G,\mathbb{Z}),\mathbb{C}^\times)\quad \text{and} \quad  H^2(G, \mathbb{Q}/\mathbb{Z})\simeq Hom(H_2(G,\mathbb{Z}),\mathbb{Q}/\mathbb{Z}).$$
Since $H_2(G,\mathbb{Z})$ is finite : $$Hom(H_2(G,\mathbb{Z}),\mathbb{C}^\times)\simeq Hom(H_2(G,\mathbb{Z}), \phi(\mathbb{Q}/\mathbb{Z}))\simeq Hom(H_2(G,\mathbb{Z}),\mathbb{Q}/\mathbb{Z}),$$
and therefore the finite groups $H^2(G,\C^\times)$ and $H^2(G,\mathbb{Q}/\mathbb{Z})$ are isomorphic. This complete the proof. Indeed, we have that $\phi_*^G$ is a surjective morphism between two finite isomorphic groups.
\end{proof}
\begin{proposition}
The morphism $\phi_*^G$ restricts to an isomorphism beteween the Bogomolov multiplier $B_0(G)$ and the group $B_\mathbb{C}(G)$.
\end{proposition}
\begin{proof}
This follows from the definitions of both groups and the fact that for an abelian subgroup $A$ of $G$ we have a commutative diagram : 
$$\begin{tikzcd}
H^2(G,\mathbb{C}^\times) \arrow{r} & H^2(A,\mathbb{C}^\times)  \\ 
 H^2(G, \mathbb{Q}/\mathbb{Z})\arrow{u}{\phi_*^G} \arrow{r}&   H^2(A, \mathbb{Q}/\mathbb{Z})\arrow{u}{\phi_*^A}
\end{tikzcd} .$$
where the vertical morphisms are isomorphisms and the horizontal morphisms are the restriction maps. The isomorphism betwee
\end{proof}

\begin{lemma}
Let $\alpha$ be a cocycle whose class lie in $B_{\mathbb{C}}(G)$. If $g$ and $h$ commute in $G$, then $\alpha(g,h)=\alpha(h,g)$. 
\end{lemma}
\begin{proof}
The restriction of $\alpha$ to the abelian group generated by $g,h$ is a coboundary $\beta$. A coboundary of an abelian group satisfies $\beta(x,y)=\beta(y,x)$ for any $x,y$ in the abelian group.
\end{proof}
\begin{proposition}
The groups $H_S^2(G,\mathbb{C}^\times)$ and $B_0(G)$ are isomorphic.
\end{proposition}
\begin{proof}
In view of the last proposition, we only need to prove that $H_S^2(G,\mathbb{C}^\times)$ is isomorphic to $B_\mathbb{C}(G)$. Take $\alpha$ a $2$-cocycle with class in $B_\mathbb{C}(G)$. We can choose $\alpha$ to be normalized. Denote by, $\mathbb{C}^\times \times_\alpha G$ the central extension of $G$ by $\mathbb{C}^\times$ associated to $\alpha$. For $x,y\in \mathbb{C}^\times$ and $g,h \in G$ :
$$(x,g)(y,h)=(xy\alpha(g,h),gh) \quad \text{and} \quad (y,h)(x,g)=(xy\alpha(h,g),hg).$$
It follows form the previous lemma and the equation that if $h$ and $g$ commute in $G$ then their lifts in $\mathbb{C}^\times \times_\alpha G$ commute. For each conjugacy class $C$ of $G$ choose a representative $g_C$ and set for $h\in G$ :
$$s(hg_Ch^{-1})=(1,h)(1,g_C)(1,h)^{-1}.$$
$s$ defines a section. Indeed, if $h_1g_Ch_1^{-1}=hg_Ch^{-1}$, we have that $h^{-1}h_1$ lies in the centralizer of $g_C$ and hence $(1,h)^{-1}(1,h_1)$ lies in the centralizer of $(1,g_C)$ by the commutativity of lifts of commuting elements we have seen. Now take $g$ and $h$ in G. We have that $g=kg_Ck^{-1}$ for some $k\in G$ and some class $C$. Hence :
\begin{align*}
s(hgh^{-1})&=s(hkg_Ck^{-1}h^{-1})\\
&=(1,hk)(1,g_C)(1,hk)^{-1}\\
&=(\alpha(h,k)^{-1},1)(1,h)(1,k)(1,g_C)(1,k)^{-1}(1,h)^{-1}(\alpha(h,k)^{-1},1)^{-1}\\
&=(1,h)s(g)(1,h)^{-1}\\
&=s(h)s(g)s(h)^{-1}
\end{align*}
where the last line is due to the fact that $s(h)$ and $(1,h)$ differ by a central element. The cocycle $\alpha$ is cohomologus to the cocycle $\beta$ given by :
$$\beta(g,h)=s(g)s(h)s(gh)^{-1}.$$
Conjugating the last equation by $s(h)$ and using that $s(h)s(gh)s(h)^{-1}=s(hg)$, we get :
$$\beta(g,h)=s(h)s(g)s(hg)^{-1}=\beta(h,g).$$
This proves that $\beta \in H_S^2(G,\mathbb{C}^\times)$ and hence also $\alpha$ lies in $H_S^2(G,\mathbb{C}^\times)$. This proves the inclusion of $B_{\mathbb{C}}(G)$ in $H_S^2(G,\mathbb{C}^\times)$. The inclusion in the other direction was already proven.
\end{proof}
It follows from the last proposition and theorem \ref{thsym} that :
\begin{theorem}
The irreducible quandle representations of a finite group $G$ over $\mathbb{C}$ are the representations of the form $\chi \cdot \rho$ for $\chi$ a character of $Conj(G)$ and $\rho$ an irreducible linear group representation of $G$, if and only if $B_0(G)=0$.
\end{theorem}
\begin{example}
Simple groups and with exceptions quasi-simple groups has trivial Bogomolov multiplier according to $($\cite{Bog}$)$.
\end{example}
\section{The envelopping group}
In this section $G$ is a \textbf{finite} group. We recall that we have denoted the envelopin group of $Conj(G)$ by $A(G)$. We have seen that we have morphism $\pi : A(G) \to G$ with kernel lying in the center of $A(G)$. The kernel of this map will be \textbf{denoted} by $Z_1$. The following is well known :
\begin{proposition}
\item[1)] The abelianisation of the envelopping group of a quandle $Q$ is the free abelian group $\mathbb{Z}Q/Inn(Q)$ on the set of orbits $Q/Inn(Q)$ of $Q$ with respect to the action of $Inn(Q)$
\item[2)] The abelianisation map $Ab : G(Q) \to \mathbb{Z}Q/Inn(Q)$ associate to a generator its class in $Q/Inn(Q)$;
\end{proposition}
\begin{proposition}\label{prop}
\item[1)]  The intersection of $Z_1$ with derived group $A(G)'$ ($ker(Ab)$ also) of $A(G)$ is equal to the torsion subgroup $Tor(Z_1)$ of $Z_1$.
\item[2)] $Z_1/Tor(Z_1)$ is isomorphic to $\mathbb{Z}G/Inn(G)$.
\end{proposition}
\begin{proof}
$Z_1$ lies in the center of $A(G)$ an has finite index. Hence, the center of $A(G)$ has finite index. It follows from Schurs theorem for the derived group that $A(G)'$ is finite. This poves that $A(G)'\cap Z_1 \subset Tor(Z_1)$. $1)$ of the proposition follows from the fact that the abelianisation of $A(G)$ is a free abelian group and hence $Tor(Z_1) \subset A(G)'$. We prove $2)$. $Z_1$ is of finite index $A(G)$ hence $Ab(Z_1)\simeq  Z_1/Tor(Z_1)$ is of finite index in the free abelian group $\mathbb{Z}G/Inn(G)$. $2)$ follows. 
\end{proof}
We will \textbf{denote} by $H_{S}^2(G,\mathbb{C}^\times)$ the subgroup of $H^2(G,\mathbb{C}^\times)$ corresponding to classes of symmetric $2$-cocycles. 
\begin{proposition}
We have $Tor(Z_1)\simeq H_S^2(G,\mathbb{C}^\times)$, equivalently $Tor(Z_1) \simeq  B_0(G)$.
\end{proposition}
\begin{proof}
We have an inflation-restrection exact sequence associated to the exact sequence $1\to Z_1\to A(G) \overset{\pi}{\to} G \to 1$ : 
$$1 \to  Hom(G,\mathbb{C}^\times)\to Hom(A(G),\mathbb{C}^\times)\overset{Res}{\to} Hom(Z_1,\mathbb{C}^\times)\to H^2(G,\mathbb{C}^\times)\overset{Inf_\pi}{\to}H^2(A(G),\mathbb{C}^\times),$$
where $Res$ is the restriction map and $Inf_\pi$ is the inflation map associated to $\pi$. Since the abelianisation of $A(G)$ is free abelian and $Z_1\cap A(G)'=Tor(Z_1)$ it is easy to see that $cocker(Res)\simeq Hom(Tor(Z_1),\mathbb{C}^\times)\simeq Tor(Z_1)$. This proves that $ker(Inf_\pi) \simeq Tor(Z_1)$. On the other hand the kernel of $Inf_\pi$ correspond to classes of irreducible projective representation of $G$ that lift via $\pi$ to linear representation of the group $A(G)$ (equivalently are induced by quandle representations of $Conj(G)$). These classes are exactly the classes that has symmetric representatives (by lemma \ref{lem}). Therfore, $ker(Inf_\pi)\simeq H_S^2(G,\mathbb{C}^\times)$. Together with the isomorphism $ker(Inf_\pi) \simeq Tor(Z_1)$ we deduce the proposition.
\end{proof}

\begin{proposition}
The following equivalent conditions hold if and only if $H_S^2(G,\mathbb{C}^\times)=0$ $($i.e. $Tor(Z_1)=0$, i.e. $B_0(G)=0$$)$.
\item[1)] The morphism $\pi \times Ab : A(G) \to G \times \mathbb{Z} G/Inn(G)$ is injective.
\item[2)] $A(G)'\simeq G'$, where $G'$ is the derived group of $G$.
\item[3)] $Z_1 \simeq \mathbb{Z} G/Inn(G)$.
\end{proposition}
\begin{proof}
This proposition follows from proposition \ref{prop} and the previous proposition.
\end{proof}
We note that in \cite{leb} it is shown that the $3$ conditions in the previous proposition are satisfied by what they call $\bar{C}$-groups.
\begin{corollary}
If $G$ is perfect and $H_S^2(G,\mathbb{C}^\times)=0$ $($equivlentaly $B_0(G)=0$$)$, then $\pi \times Ab : A(G) \to G\times \mathbb{Z} G/Inn(G)$ is an isomorphism.
\end{corollary}
If $G$ is a perfect group with trivial Bogolomov multiplier or $G$ is a Schur cover of finite perfect group, then $G$ satisfies the hypothesis of the corollary.

 \end{document}